\documentclass{amsart}
\usepackage{amsmath}
\usepackage{amssymb}
\usepackage{amsthm}
\usepackage{mathrsfs}

\newtheorem{theorem}{Theorem}[section]
\newtheorem{lemma}[theorem]{Lemma}

\theoremstyle{definition}
\newtheorem{definition}[theorem]{Definition}

\newtheorem{Notation}[theorem]{Notation}
\newtheorem{corollary}[theorem]{Corollary}

\theoremstyle{remark}
\newtheorem{remark}[theorem]{Remark}

\numberwithin{equation}{section}

\begin{document}

\title{RINGS AND SUBRINGS OF CONTINUOUS FUNCTIONS WITH COUNTABLE RANGE}

\author{ Sudip Kumar Acharyya}
\address{Department of Pure Mathematics, University of Calcutta, 35, Ballygunge Circular Road, Kolkata - 700019, INDIA} 
\email{sdpacharyya@gmail.com}

\author{Rakesh Bharati}
\address{Department of Pure Mathematics, University of Calcutta, 35, Ballygunge Circular Road, Kolkata - 700019, INDIA} 
\email{bharti.rakesh292@gmail.com}

\thanks{The second author acknowledges financial support from University Grand Commission, New Delhi, for the award of research fellowship (File No. 16-9(June 2018)/2019 (NET/CSIR))}

\author{A. Deb Ray }
\address{Department of Pure Mathematics, University of Calcutta, 35, Ballygunge Circular Road, Kolkata - 700019, INDIA} 
\email{debrayatasi@gmail.com}

\begin{abstract}
	Intermediate rings of real valued continuous functions with countable range on a Hausdorff zero-dimensional space $X$ are introduced in this article. Let $\Sigma_c(X)$ be the family of all such intermediate rings $A_c(X)$'s which lie between $C_c^*(X)$ and $C_c(X)$. It is shown that the structure space of each $A_c(X)$ is $\beta_0X$, the  Banaschewski compactification of $X$. $X$ is shown to be a $P$-space if and only if each ideal in $C_c(X)$ is closed in the $m_c$-topology on it. Furthermore $X$ is realized to be an almost $P$-space when and only when each maximal ideal/ $z$-ideal in $C_c(X)$ becomes a $z^0$-ideal. Incidentally within the family of almost $P$-spaces, $C_c(X)$ is characterized among all the members of $\Sigma_c(X)$ by virtue of either of these two properties. Equivalent descriptions of pseudocompact condition on $X$ are given via $U_c$-topology, $m_c$-topology and norm on $C_c(X)$. The article ends with a result which essentially says that $z^0$-ideals in a typical $A_c(X)$ $\in$ $\Sigma_c(X)$ are precisely the contraction of $z^0$-ideals in $C_c(X)$.
\end{abstract}

\maketitle

\section{Introduction}
\noindent In what follows $X$ stands for a completely regular Hausdorff topological space and $C(X)$ as usual denotes the ring of all real valued continuous functions on $X$. $C^*(X)$ designates the subring of $C(X)$ containing all those members which are bounded over $X$. Suppose $C_c(X)$ is the subset of $C(X)$ consisting of those functions $f$ for which $f(X)$ is a countable subset of $\mathbb{R}$ and $C_c^*(X)$=$C_c(X)$$\cap$$C^*(X)$. It is well known that $C_c(X)$ (respectively $C_c^*(X)$) is a subring as well as a sublattice of $C(X)$ (respectively $C^*(X)$). These two rings $C_c(X)$ and $C_c^*(X)$ have received the attention of a few experts in this area only recently. We refer to the reader the articles \cite{ref5},\cite{ref11},\cite{ref20},\cite{ref15} in this connection. A natural expectation has cropped up as a bye-product of these recent investigations that, there is a hidden interplay existing between the topological structure of $X$ and the ring and the lattice structure of $C_c(X)$ and $C_c^*(X)$. To study this interaction in an efficient manner the authors in~\cite{ref14} have already discovered that one can stick to a well chosen class of spaces viz. the zero-dimensional Hausdorff topological space $X$. Indeed it is proved in (\cite{ref14}, Theorem 4.6) that starting from any topological space $X$ (not necessarily even completely regular), one can construct a Hausdorff zero-dimensional space $Y$ such that the ring $C_c(Y)$ is isomorphic to the ring $C_c(X)$. This may be called the analogous fact for its classical antecedent in the theory of $C(X)$ which says that any topological space $X$ can give rise to a completely regular Hausdorff space $Y$ for which $C(X)$ is isomorphic to $C(Y)$ (\cite{ref13}, Theorem 3.9). Therefore in the study of $C_c(X)$ and $C_c^*(X)$ vis-a-vis the space $X$ the ambient topological space $X$ may well chosen to be Hausdorff and zero-dimensional in the sense that clopen sets make base for the topology on $X$. We will stick to this convention throughout this article. Furthermore an ideal $I$ unmodified in any ring R in this paper will always stand for a proper ideal.

It is a standard result in the theory of Rings of Continuous functions that the structure space of $C(X)$ and $C^*(X)$ are both $\beta X$, the Stone-$\check{C}$ech compactification of $X$ (7N, \cite{ref13}). As a countable counterpart of the result, it is proved in (\cite{ref5}, Remark~3.6) that the structure space of $C_c(X)$ is $\beta_0X$, the largest zero-dimensional compactification of a zero-dimensional Hausdorff space $X$, also known as Banaschewski compactification of $X$. The structure space of a commutative ring $R$ with unity stands for the set of all maximal ideals of $R$ equipped with the familiar hull-kernel topology. In the present article we have initiated the study on intermediate rings viz those rings that lie between $C_c^*(X)$ and $C_c(X)$. Let $\Sigma_c(X)$ stand for the aggregate of all such intermediate rings. In section 2 of the present article we establish that if $A_c(X)$ $\in $ $\Sigma_c(X)$ then the structure space of $A_c(X)$ is also $\beta_0X$ (Theorem \ref{t-2.7}). This generalizes the Proposition mentioned in \cite{ref5}. This is incidentally the first important technical result in this article. A space X is termed as a $CP$-space in \cite{ref14} if the ring $C_c(X)$ is regular in the sense of Von-Neumann and several equivalent versions of this property are recorded in (\cite{ref14}, Theorem 5.8). These are natural counterparts of the corresponding equivalent descriptions of a $P$-space in the classical setting of $C(X)$ as mentioned in (4J, \cite{ref13}). It is also proved in the same article (\cite{ref14}, Corollary 5.7) that a zero-dimensional space $X$ is a $CP$-space if and only if it is a $P$-space. In section 3 of the present article we introduce $m_c$-topology on $C_c(X)$ as a counterpart for the present set up of the well known $m$-topology on $C(X)$, introduced longtime back by Hewitt in 1948 \cite{ref23}. We prove that if $I$ is an ideal of $C_c(X)$ , then the closure of $I$ in the $m_c$-topology coincides with the intersection of all the maximal ideals of $C_c(X)$ which contain $I$ (Theorem \ref{t-3.8}). From this it follows that a zero-dimensional space $X$ is $P$-space if and only if each ideal in $C_c(X)$ is closed in $m_c$-topology (Theorem \ref{t-3.10}). We further  establish that if $A_c(X)$ $\in $ $\Sigma_c(X)$ is properly contained in $C_c(X)$ then it is never Von-Neumann regular (Theorem\ref{t-3.17}). Thus within the class of $P$-spaces $X$, $C_c(X)$ is characterized amongst all the intermediate rings by the property that it is Von-Neumann regular.

A Tychonoff space $X$ is called almost $P$ if the interior  of each non empty zero set in $X$ is open. These spaces are introduced in \cite{ref17} as a generalization of $P$-spaces. In section 4 of this article we make some query about when a zero-dimensional space $X$ becomes an almost $P$-space. We establish that $X$ is almost $P$ if and only if each maximal ideal of $C_c(X)$ is a $z^0$-ideal and this happens when and only when each $z$-ideal in $C_c(X)$ becomes a $z^0$-ideal (Theorem \ref{t-4.10}). It turns out that within the class of almost $P$-space $X$ , $C_c(X)$ is the unique ring amongst all the intermediate rings that lie between $C_c^*(X)$ and $C_c(X)$ which enjoys either of these two properties (Theorem \ref{t-4.11},\ref{t-4.12}).

A space $X$ is called pseudocompact if $C(X)=C^*(X)$. It is established by the authors in (\cite{ref15}, Theorem 6.3) that a zero-dimensional space $X$ is pseudocompact if and only if  $C_c(X)=C_c^*(X)$ . In section 5 of this article we find out a few equivalent versions of pseudocompactness in terms of both the $m_c$-topology on  $C_c(X)$ and $U_c$-topology on  $C_c(X)$ (Theorem \ref{t-5.2}, \ref{t-5.3}). The $U_c$-topology on  $C_c(X)$ may be called the countable counterpart of the well known $U$-topology or the topology of uniform convergence on  $C(X)$ (See 2M, 2N, \cite{ref13}).

In section 6 of this article we examine when do a few chosen subrings of $C_c(X)$ become Noetherian/ Artinian (Theorem \ref{t-6.4}). It follows as special cases that a zero-dimensional space $X$ is finite if and only if $C_c(X)$ is Noetherian if and only if $C_c(X)$ is Artinian. Furthermore a locally compact zero-dimensional space $X$ is seen to be finite if and only if $C_c(X) $$\cap$$  C_K(X) $  becomes Noetherian/Artinian if and only if $C_c(X)\cap C_\infty(X)$ becomes Noetherian/ Artinian. Here $C_\infty(X)$ stands for the rings of all real valued continuous functions on $X$ which vanish at infinity and $C_K(X)$ is the subring of $C_\infty(X)$ containing those functions which have compact support.

In the final section 7 of this article we give an explicit formula for $z^0$-ideals in a typical intermediate ring $A_c(X)$ $\in \Sigma_c(X)$ (Theorem \ref{t-7.1}). From this it follows that $z^o$-ideals of $A_c(X)$, in particular $z^0$-ideals of $C_c(X)$ or $C_c^*(X)$ are the contraction of $z^0$-ideals in $C(X)$.

\section{STRUCTURE SPACES OF INTERMEDIATE RINGS}

We recall from (7M \cite{ref13})  that if $A$ is a commutative ring with unity and $\mathcal{M}(A)$  the set of all  maximal ideals of A and for each $a\in A$ if we set $\mathcal{M}_a$=$\{M\in\mathcal{M}(A):a\in M\}$, then the family \{$\mathcal{M}_a$: $a\in A$\} turns out to be a closed base for the hull-kernel topology on $\mathcal{M}(A)$. For any $\mathcal{M}_0 $ $\subseteq $ $\mathcal{M}(A)$ the closure of $\mathcal{M}_0$=\{$M\in \mathcal{M}(A):M\supset\cap\mathcal{M}_0$\}. $\mathcal{M}(A)$ equipped with this topology known as the structure space of $A$ is a compact $T_1$ topological space and is Hausdorff if and only if given any two distinct maximal ideals $M_1$ , $M_2$ in $A$, there exist points $a_1$, $a_2$ in $A$ such that $a_1\notin M_1$ and $a_2\notin M_2$ and $a_1a_2\in\cap\mathcal{M}(A)$. In what follows we will let $A_c(X)$ stand for a typical intermediate ring lying between the two rings $C_c^*(X)$ and $C_c(X)$. Suppose $Max (A_c(X))$ denotes the structure space of $A_c(X)$.

\begin{theorem}\label{t-2.1}
$Max(A_c(X))$ is a (compact) Hausdorff space.
\end{theorem}
\begin{proof}We shall prove the Hausdorffness of $Max(A_c(X))$ only. For any $f\in A_c(X)$, set $\mathcal{Z}_A(f)$=$\{Z\in Z_c(X):$ there exists $g\in A_c(X)$ such that for each $x\in X\setminus Z, f(x)g(x)=1\}$. Here $Z_c(X)=\{Z(f):f\in C_c(X\}$, the family of all zero sets in $X$ of functions lying in $C_c(X)$. For any ideal $I$ in $A_c(X)$, let $\mathcal{Z}_A[I]=\bigcup\limits_{f\in I} \mathcal{Z}_A(f)$. Then it can be proved by following the technique adopted in \cite{ref10},\cite{ref19},\cite{ref20} that $\mathcal{Z}_A(f)$ and $\mathcal{Z}_A[I]$ are both $z_c$-filter on $X$. A $z_c$-filter on $X$ is a subfamily of $Z_c(X)-\{\emptyset\}$ which is closed under finite intersection and formation of supersets (see \cite{ref14}). Furthermore if $\mathcal{F}$ is a $z_c$-filter on X, then it can be checked by using the methods in \cite{ref10},\cite{ref19},\cite{ref20} that $\mathcal{Z}_A^{-1}[\mathcal{F}]$= $\{f\in A_c(X):\mathcal{Z}_A(f)\subseteq\mathcal{F}\}$ is a (proper) ideal in $A_c(X)$. Now let $M_1$ and $M_2$ be two distinct members of $A_c(X)$. It is sufficient to produce $h_1$, $h_2$ in $\mathcal{M}_c(X)$ with $h_1\notin M_1$, $h_2\notin M_2$ such that $h_1h_2=0$. To this end we assert that there exists $Z_1\in \mathcal{Z}_A[M_1]$, $Z_2\in \mathcal{Z}_A[M_2]$ with $Z_1\cap Z_2= \emptyset$. For otherwise each member of $\mathcal{Z}_A[M_1]$ meets any member of $\mathcal{Z}_A[M_2]$ and hence $\mathcal{Z}_A[M_1] \cup\mathcal{Z}_A[M_2]$ becomes a subfamily of $Z_c(X)$ with finite intersection property. Consequently there exists a $z_c$-filter $\mathcal{F}$ on $X$ such that $\mathcal{Z}_A[M_1] \cup\mathcal{Z}_A[M_2] \subseteq \mathcal{F}$ which yields that $M_1\cup M_2\subseteq\mathcal{Z}_A^{-1} [\mathcal{F}]$= a proper ideal in $A_c(X)$, a contradiction since $M_1$ and $M_2$ are distinct maximal ideals in $A_c(X)$. So choose $f\in M_1$ and $g\in M_2$ such that $Z_1$$\in$ $\mathcal{Z}_A(f)$, $Z_2$$\in$ $\mathcal{Z}_A(g)$ and $ Z_1\cap Z_2=\emptyset$. This means that there exist $f_1, g_1\in A_c(X)$
such that for each $x\in X-Z_1$, $f(x)f_1(x)=1$ and for any $x\in X-Z_2$, $g(x)g_1(x)=1$. Now since $ff_1\in M_1$ and $gg_1\in M_2$, it follows that $1-ff_1\notin M_1$ and $1-gg_1\notin M_2$. Since $(X-Z_1)\cup(X-Z_2)= X-(Z_1\cap Z_2)=X-\emptyset=X$, it implies that $(1-ff_1).(1-gg_1)=0$.
\end{proof}

 We set for any $x\in X$, $M_{A,x}=\{f\in A_c(X):f(x)=0\}$. Then it is easy to prove on applying the first isomorphism theorem of algebra, taking care of the presence of constant functions in $A_c(X)$ that the complete list of fixed maximal ideals of $A_c(X)$ is given by $\{M_{A,x} : x \in X\}$. An ideal $I$ in $A_c(X)$ is called fixed if there exists a point on X at which all the functions in $I$ vanish. For any $f \in A_c(X)$ we set $(\mathcal{M}_A)_f=\{M\in Max (A_c(X)):f\in M\}$. Then  $\{(\mathcal{M}_A)_f:f\in A_c(X)\}$ is the family of basic closed sets in the structure space $Max(A_c(X))$ of $A_c(X)$. For $f\in A_c(X)$ and $x\in X$ , $x\in Z(f)$ if and only if $M_{A,x}\in (\mathcal{M}(A))_f \cap\{M_{A,y}:y \in X\}$. On the other hand since $X$ is zero-dimensional it follows from Proposition 4 in \cite{ref5} that $\{Z(f):f\in C_c^*(X)\}=\{Z(f):f\in A_c(X)\}= \{Z(f):f\in C_c(X)\}$ constitutes a base for the closed sets of $X$. These two facts therefore yield that the map $\psi _A:X\rightarrow Max(A_c(X))$ given by $\psi _A(x)=M_{A,x}$ which is obviously one-to-one exchanges the basic closed sets of the space $X$ and the subspace $\psi _A(X)$ of $Max(A_c(X))$. Furthermore the closure of $\psi _A(X)$ in $Max(A_c(X))$ is given by $\{M\in Max(A_c(X)):M\supseteq\bigcap\psi_A(X)\}=\{M\in Max(A_c(X)):M\supseteq\{0\}\}=Max(A_c(X))$, thus demonstrating that $\psi_A(X)$ is dense in $ Max(A_c(X))$. The above observations therefore lead to the following proposition.

\begin{theorem} \label{t-2.2}
The pair $(\Psi_A,Max (A_c(X)))$ is a Hausdorff compactification of $X$ in the following sense, which we reproduce from the monograph \cite{ref12}.
\end{theorem}
 \begin{definition}
 	A (Hausdorff) compactification of a Tychonoff space $X$ stands for a pair ($\alpha $, $\alpha X$), where $\alpha X$ is a compact Hausdorff space and $\alpha:X$$\rightarrow$$\alpha X$ is a topological embedding with $\alpha(X)$ dense in $\alpha X$. For simplicity we often write $\alpha X$ instead of ($\alpha$, $\alpha X$). Let $K(X)$ be the family of all Hausdorff compactifications of $X$.
 \end{definition}
\begin{definition}
	For $\alpha X$, $\gamma X \in K(X) $, we write  $\alpha X\geqq\gamma X$ if there is a continuous map $t:\alpha X\rightarrow \gamma X	$ with the property $t\circ\alpha=\gamma$. If in this definition $'t'$ is a homeomorphism then we say that $\alpha X$ is topologically equivalent to $\gamma X$ and we write $\alpha X$$\approx$$\gamma X$. It can be proved without difficulty that for $\alpha X$,  $\gamma X$$\in$ $K(X)$, $\alpha X$$\approx$$\gamma X$ when and only when $\alpha X\geqq\gamma X$ and $\gamma X$$\geqq$$\alpha X$. Furthermore $(K(X),\geqq)$ becomes a complete upper semilattice, which has definitely then a largest member,which is incidentally $\beta X$ the Stone-$\check{C}$ech compactification of $X$. If in addition $X$ is zero-dimensional then there is a largest zero-dimensional member of $K(X)$, designated by $\beta_ 0X$, called the Banaschewski compactification of $X$. For more information on these topics see \cite{ref18}.
\end{definition}

\begin{definition}
	For a zero-dimensional space $X$, $\alpha X \in K(X)$ is said to enjoy C-extension property if given any compact Hausdorff zero-dimensional space $Y$ and a continuous map $f:X\rightarrow Y$ there exists a unique continuous map $f^\alpha:\alpha X\rightarrow Y$ such that $f^\alpha \circ \alpha=f$.
\end{definition}

It is clear from the above definition that if $\alpha X \in K(X)$ possesses $C$-extension property then $\alpha X\geqq \beta_0X$ and if in addition $\alpha X$ is zero-dimensional then $\beta_0X\geqq \alpha X$ and consequently $\alpha X \approx\beta_0X$. We need the following subsidiary result before stating the first principal technical result of this section.
\begin{theorem} \label{t-2.6}
	Let X be zero-dimensional and $A_c(X) \in \Sigma_c(X)$. Then given $f\in A_c(X)$, there exists an idempotent $e$ in $A_c(X)$ such that $e$ is multiple of $f$ and $(1-e)$ is a multiple of $(1-f)$ in this ring. $\textnormal[$A special case of this result with $A_c(X)=C_c(X)$ is proved in Remark 3.6 in \cite{ref5} $\textnormal]$.
\end{theorem} 

\begin{proof}
	There exists $r,\ 0<r<1$ such that $r\notin f(X)$. Let $W=f^{-1}(-\infty,r)=f^{-1}((-\infty,r])$. So $W$ and $X-W$ are both clopen sets in $X$. The function $e:X\rightarrow R$ defined by the rule : $e(W)=0$ and $e(X-W)=1$ is clearly an idempotent in the ring $A_c(X)$. Define the functions $h:X\rightarrow R$ and $k:X\rightarrow R$ as follows: $h(W)=0$ and $h(x)=\frac{1}{f(x)}$ if $x\in X-W$. $k(X-W)=0$ and $k(x)=\frac{1}{1-f(x)}$ if $x\in W$. Clearly $h$ and $k$ are both bounded functions in $C_c(X)$ and hence both are members of the ring $A_c(X)$. It is easy to see that $e=h.f$ and $1-e=k(1-f)$.
\end{proof}
\begin{theorem} \label{t-2.7}
	$Max (A_c(X))$ is a (compact Hausdorff) zero-dimensional space. Furthermore the pair $(\Psi_A,Max(A_c(X))$ is topologically equivalent to $\beta_0 X$. If in addition $X$ is strongly zero-dimensional meaning that $\beta X$ is zero-dimensional, then $(\Psi_A, Max(A_c(X)))$ is topologically equivalent to $\beta X$.
\end{theorem}

\begin{proof}
	We first prove (only) the zero-dimensionality of $Max(A_C(X))$, because of Theorem \ref{t-2.1}. We recall the notation that for any $f\in A_c(X)$, $(\mathcal{M}_A)_f=\{M\in Max(A_c(X)):f\in M\}$ . So let $M\in Max(A_c(X))$ and $f\in A_c(X)$ be such that $M\in Max(A_c(X))\setminus (\mathcal{M}_A)_f$. It suffices to find out a clopen set in $Max(A_c(X))$ which contains $M$ and is contained in  $Max(A_c(X))\setminus(\mathcal{M}_A)_f$. We first observe that $M\notin (\mathcal{M}_A)_f$  implies that $f\notin M$ which in turn implies that there exist $h\in A(_c(X))$ and $g\in M$ such that $1-g=hf$. By Theorem \ref{t-2.6}, there exists an idempotent $'e'$ in $A_c(X)$ such that $e$ is a multiple of $g$ and $(1-e)$ is a multiple of $(1-g)$ in the ring $A_c(X)$. Since $g\in M$, this implies that $e\in M$, in other words $M\in (\mathcal{M}_A)_e$. On the otherhand if  $N\in (\mathcal{M}_A)_f$ then $f\in N$, hence $1-g=hf\in N$ consequently $1-e\in N$ and therefore $e\notin N$ (as $N$ is a maximal ideal in $A_c(X)$) which means that $N\notin (\mathcal{M}_A)_e$. Thus  we get that $M\in (\mathcal{M}_A)_e \subseteq Max(A_c(X))\setminus(\mathcal{M}_A)_f$. We now assert that $(\mathcal{M}_A)_e=Max(A_c(X))\setminus(\mathcal{M}_A)_{1-e} $ and have $(\mathcal{M}_A)_e$ is clopen in  $Max(A_c(X)$. Indeed if $M\in (\mathcal{M}_A)_e$ then $e\in M$, which implies that $1-e\notin M$ and hence $M\notin (\mathcal{M}_A)_{1-e}$ $i.e;$ $M\in Max(A_c(X))\setminus(\mathcal{M}_A)_{1-e}$. Thus $ (\mathcal{M}_A)_e\subseteq Max(A_c(X))\setminus  (\mathcal{M}_A)_{1-e}$. From symmetry it follows that, as $(1-e)$ is an idempotent of $A_c(X)$.  $(\mathcal{M}_A)_{1-e} \subseteq Max(A_c(X))$$\setminus$$(\mathcal{M}_A)_e$, hence  $(\mathcal{M}_A)_{e}=Max(A_c(X))$$\setminus$$(\mathcal{M}_A)_{1-e}$.
	
	Now that we have proved that $ Max(A_c(X))$ is zero-dimensional, to prove the second part of the present theorem, it is sufficient to prove that $(\Psi_A,Max(A_c(X))$ enjoys the $C$-extension property. So let $Y$ be a compact Hausdorff zero-dimensional space and $f:X\rightarrow Y$ a continuous map. It is sufficient to define a continuous map $f^\mathcal{M} :Max(A_c(X))\rightarrow Y$ with the following property: $f^\mathcal{M}\circ \Psi_A=f$. To that end choose  $M\in Max(A_c(X))$ i.e; $M$ is a maximal ideal in $A_c(X)$. Set $\widetilde{M}=\{g\in C_c(Y):g\circ f\in M\}$. Note that if $g\in C_c(Y)$ then $g\circ f\in C_c(X)$. Furthermore since $Y$ is compact and $g\in C_c(Y)$ then $g(Y)$ is a bounded subset of $\mathbb{R}$, consequently $(g\circ f)(X)$ is a bounded subset of $\mathbb{R}$ and hence ($g\circ f)\in C_c^*(X)$ and therefore $g\circ f\in A_c(X)$. Thus the definition of $\widetilde{M}$ is without any ambiguity. Since $M$ is a maximal ideal of $ A_c(X)$ it follows that $\widetilde{M}$ is a prime ideal of $C_c(Y)$. Now it is already proved in (\cite{ref14}, Corollary 2.14) that every prime ideal in $C_c(Y)$ is contained in a unique maximal ideal. Thus $\widetilde{M}$ extends to a unique maximal ideal in$C_c(Y)$ which is fixed because $Y$ is compact. Thus there exists a unique point $y\in Y$ such that for each $g\in \widetilde{M}$, $g(y)=0$ and hence $\bigcap \limits_{g\in \widetilde{M}} Z(g)=\{y\}$. We set $f^\mathcal {M}(M)=y$. Thus \{$f^\mathcal {M}(M)$\}=$\bigcap \limits_{g\in \widetilde{M}} Z(g)$  ...(1).
	 We note that if $x\in X$ and $g\in \widetilde{M}_{A,x}$, then $g\circ f\in M_{A,x}$ and hence $(g\circ f)(x)=0$, which implies that $f(x)\in Z(g)$. This proves that $\bigcap \limits_{g\in \widetilde{M}_{A,x}} Z(g)=\{f(x)\}$. This implies in view of the definition (1) above that $f^\mathcal {M}(M_{A,x})=f(x)$, in other words: $f^\mathcal{M}\circ \Psi_A(x)=f(x)$. Thus $f^\mathcal{M}\circ \Psi_A=f$. To ensure the continuity of the map $f^\mathcal{M}:Max(A_c(X)\rightarrow Y$ defined in (1) at an arbitrary $M\in Max(A_c(X))$, let $W$ be a neighbourhood of $f^\mathcal{M}(M)$ in the space $Y$. Since $Y$ is zero-dimensional, each neighbourhood of a point $'y'$ in this space contains a co-zero set neighbourhood of $y$ of the form $Y\setminus Z(g_1)$
 for some $g_1\in C_c(Y)$ and also a zero set neighbourhood of $y$ of the form $Z(g_2)$ for an appropriate $g_2\in C_c(Y)$ (see Proposition 4.4, \cite{ref14}). Thus there exist $g_1,g_2\in C_c(Y)$ such that $f^\mathcal{M}\in Y\setminus Z(g_1)\subset Z(g_2)\subset W$ ...(2).
 As $f^\mathcal{M}(M)\notin Z(g_1)$, it follows from (1) that $g_1\notin \widetilde{M}$ which implies that $g_1\circ f\notin M$, in other words $M\notin (\mathcal{M}_A)_{g_1\circ f}$. Thus $Max (A_c(X))\setminus(\mathcal{M}_A)_{g_1\circ f}$ becomes an open neighbourhood of $M$ in the space $Max(A_c(X))$. We assert that $f^\mathcal{M}(Max (A_c(X))\setminus(\mathcal{M}_A)_{g_1\circ f}))$ $\subseteq W$ and this settles the continuity of  $f^\mathcal{M}$ at the point $M$.\\
\noindent Proof of the last assertion: Let $N\in (Max (A_c(X))\setminus (\mathcal{M}_A)_{g_1\circ f}$, then $g_1\circ f\notin N$, hence $g_1\notin \widetilde{N}$. Since $g_1g_2=0$ as is evident from the relation (2) above and $\widetilde{N}$ is a prime ideal in C(Y), it follows therefore that $g_2\in \widetilde{N}$. This implies in view of the relation (1) that $f^\mathcal{M}(N)\in Z(g_2)$ and hence from (2) we get that $f^\mathcal{M}(N)\in W$.
  
   The part three of the theorem follows from the simple observation that if $\beta X$ is zero-dimensional , then $\beta_0 X \geqq \beta X$ and consequently $\beta_0 X \approx \beta X$.
	
\end{proof}

\section{$P$-spaces $X$ versus the $m_c$-topology on $C_c(X)$}

\begin{Notation}
	For any $g\in C_c(X)$ and a positive unit $u$ of this ring set $M(g,u)=\{f\in C_c(X):|f(x)-g(x)|<u(x)\ for\ each\ x\in X\}$. Then it needs a routine calculation to conclude that $\mathcal{B}=\{M(g,u):g\in C_c(X)$, $u$ a positive unit of $ C_c(X)\}$ is an open base for some topology, which we call the $m_c$-topology on $C_c(X)$. It is also not at all hard to show by employing stereotyped routine arguments that $C_c(X)$ with $m_c$-topology is a topological ring as well as a topological vector space over $\mathbb{R}$. Let $U$ stand for the set of all units in $C_c(X)$. Then for each $u\in U$, it is easy to prove that $M(u,\frac{1}{2}|u|)\subseteq U$. It follows that $U$ is an open set in $C_c(X)$ in the $m_c$-topology. It is a standard result that in a topological ring the closure of an ideal is either an ideal or the whole of the ring (2M1, \cite{ref13}). This implies that if $I$ is a proper ideal of $C_c(X)$ then the closure of $I$ in the $m_c$-topology is also a proper ideal in $C_c(X)$. We therefore get the following result:
\end{Notation}
\begin{theorem} \label{t-3.2}
	Each maximal ideal in $C_c(X)$ is closed in the $m_c$-topology.
\end{theorem}

	Before proceeding further in this technical section on $m_c$-topology on $C_c(X)$ we recall that the structure space of $C_c(X)$ is $\beta_0 X$. Hence the maximal ideals of $C_c(X)$ can be indexed by virtue of the points of $\beta_0X$. Indeed the complete list of maximal ideals in $C_c(X)$ is given in (\cite{ref5}, Theorem 4.2) by the family \{$M_c^p:p\in \beta_0X$\}, where $M_c^p=\{f \in C_c(X): p\in cl_{\beta_0X}Z(f)\}$. This is the $C$-analogue of the well known Gelfand-Kolmogoroff theorem (\cite{ref13}, Theorem 7.3).

\begin{Notation}
	
 For any ideal $I$ in $C_c(X)$ set $Q_c(I)=\{p\in\beta_0X:M_c^p\supseteq I\}$. Then the following result turns out as a simple consequence of the above formula for the maximal ideals $M_c^p$'s in $C_c(X)$.
	
\end{Notation}

	\begin{theorem}\label{t-3.4}
	$Q_c(I)=\bigcap_{f\in I}    cl_{\beta_0X}Z(f)$ , which is set of all cluster points of the $z_c$-ultrafilters $Z(I)$ in the space $\beta_0X$.
	
	\end{theorem}
	
	We need to use the following three subsidiary results to prove the first important technical result in this section.
	\begin{theorem} \label{t-3.5}
		Let $f\in C_c(X)$ and $I$ be an ideal in $C_c(X)$ such that $cl_{\beta_0X}$$Z(f)$ is a neighbourhood of $Q_c(I)$ in $\beta_0X$. Then $f\in I$.
	\end{theorem}
\begin{proof}
	The hypothesis tells that there exists an open subset $W$ of $\beta_0X$ such that $cl_{\beta_0X}Z(f)\supseteq W\supseteq Q_c(I)$. We can rewrite this relation in view of Theorem \ref{t-3.4} in the manner: $cl_{\beta_0X}Z(f)\supseteq W\supseteq \bigcap \limits_{f\in I}cl_{\beta_0X} Z(f)$ . This implies that $\beta_0X\setminus cl_{\beta_0X}Z(f)\subseteq\beta_0X\setminus W\subseteq\bigcup \limits_{f\in I}(\beta_0X\setminus cl_{\beta_0X}Z(f))$. Since the closed subset $\beta_0X\setminus W$ of $\beta_0X$ is compact, the last relation yields : $\beta_0X$$\setminus$$cl_{\beta_0X}Z(f)\subseteq\beta_0X\setminus W$$\subseteq\beta_0X\setminus\bigcap \limits_{i=1}^{n} cl_{\beta_0X}Z(f_i)$ for a suitable finite subset \{$f_1,f_2...f_n$\} of $I$. Consequently we have $cl_{\beta_0X}Z(f)\supseteq W$$\supseteq\bigcap \limits_{i=1}^{n} cl_{\beta_0X}Z(f_i)$, which further implies that $cl_{\beta_0X}Z(f)\cap X\supseteq W\cap X\supseteq Z(\sum\limits_{i=1}^{n} f_i^2)=Z(h) $ say, writing $h=f_1^2+f_2^2+...+f_n^2$. The last relation says that with $h\in I$, $Z(f)$ is a neighbourhood of $Z(h)$ in the space $X$. It follows from Lemma 2.4 in \cite{ref14} that $f$ is a multiple of $h$ in the ring $C_c(X)$. Since $h\in I$, we have $f\in I$.

\end{proof}

\begin{theorem}\label{t-3.6}
	Given $g\in C_c(X)$ and a positive unit $u$ in this ring, there exists $f\in C_c(X)$ such that $|g-f|\leq u$ and $cl_{\beta_0X}Z(f)$ is a neighbourhood of $cl_{\beta_0X}Z(g)$ in the space $\beta_0X$.
\end{theorem}
\begin{proof}
	Let the map $f:X\rightarrow \mathbb{R}$ be defined as follows: $$f(x)=\begin{cases}
	0 & if \ |g(x)|\leq u(x)\\
	g(x)+u(x) & if \ g(x)\leq-u(x)\\
	g(x)-u(x) & if \ g(x)\geq u(x)
	\end{cases}$$
	
	It is clear that $f$ is a continuous function and of course $f\in C_c(X)$. It is easily seen that $|f-g|\leq u$ on $X$. Let $F=\{x\in X:|g(x)|\geq u(x)\}$. Then $F\in Z_c(X)$ so that we can write $F=Z_c(h)$ for some $h\in C_c(X)$. Hence $Z(g)\subseteq X\setminus Z(h)$$\subseteq Z(f)$. This implies that $Z(g)\cap Z(h)=\emptyset $ and  $Z(f)\cup Z(h)=X$. From this it follows that $cl_{\beta_0X}Z(g)\cap cl_{\beta_0X}Z(h)$=$\emptyset$ and $cl_{\beta_0X}Z(g)\cup cl_{\beta_0X}Z(f)=cl_{\beta oX}X=\beta_0X$ [see Proposition 3.2 and Proposition 3.3 in \cite{ref5}]. This further yields:\\
	$cl_{\beta_0X}Z(f)\supseteq \beta_0X\setminus cl_{\beta_0X} Z(h) \supseteq cl_{\beta_0X}Z(g)$
	this shows that $cl_{\beta_0X}Z(f)$ is a neighbourhood of $cl_{\beta_0X}Z(g)$ in the space $\beta_0X$.
\end{proof}

Define as in 7Q \cite{ref13}, for an ideal $I$ in $C_c(X)$.
\\
$\overline{I}=\bigcap \{M_c^p:M_c^p\supseteq I\}$= the intersection of all maximal ideals in $C_c(X)$ which contain $I$.

\begin{theorem} \label{t-3.7}
	For any ideal $I$ in $C_c(X)$ $\overline{I}$ is a closed ideal in the $m_c$-topology and $\overline{I}=\{g\in C_c(X): cl_{\beta_0X}Z(g)\supseteq Q_c(I)\}$.
\end{theorem}

\begin{proof}
	It follows immediately from Theorem \ref{t-3.2} that $\overline{I}$ is a closed ideal in $C_c(X)$ in the $m_c$-topology.  Let $g$ $\in \overline{I}$, choose $x\in Q_c(I) $, then $M_c^x \supseteq I$. Consequently $g\in M_c^x$ and hence $x\in cl_{\beta_0X}Z(g)$. This implies that $Q_c(I) \subseteq cl_{\beta_0X}Z(g)$. To prove the reverse inclusion relation let $g\in C_c(X)$ be such that  $Q_c(I) \subseteq cl_{\beta_0X}Z(g)$. Let $M_c^p$ be any maximal ideal in $C_c(X)$ containing $I$, $p\in \beta_0X$. Then $p\in Q_c(I)$ consequently $p\in cl_{\beta_0X} Z(g) $ hence $g\in M_c^p$. Thus $g\in \overline{I}$.
\end{proof}
\begin{theorem}\label{t-3.8}
	For any ideal $I$ in $C_c(X)$, $\overline{I}$ is essentially the closure of $I$ in the $m_c$-topology.
\end{theorem}
\begin{proof}
	It follows from the first part of Theorem \ref{t-3.7} that the closure of $I$ in the $m_c$-topology is contained in $\overline{I}$. To prove the reverse containment let $g\in \overline{I}$ and $u$ be a positive unit of $C_c(X)$. It suffices to produce an $h\in I$ such that $|g-h|\leq u$. Indeed from Theorem \ref{t-3.6} there exists an $h\in C_c(X)$ with $|g-h|\leq u$ such that $cl_{\beta_0X}Z(h)$ is a neighbourhood of  $cl_{\beta_0X}Z(g)$ in the space $\beta_0X$. But $g\in \overline{I}$ implies by Theorem \ref{t-3.7} that $cl_{\beta_0X} Z(g)\supseteq Q_c(I)$. Consequently $cl_{\beta_0X}Z(h)$ becomes a neighbourhood of $Q_c(I)$ in $\beta_0X$. Hence we get from Theorem \ref{t-3.5} that $h\in I$.
	\end{proof}
\begin{corollary}\label{c-3.9}
	An ideal in $C_c(X$) is closed in the $m_c$-topology if and only if it is the intersection of all the maximal ideals in $C_c(X)$ which contain it.
\end{corollary}
\begin{theorem}\label{t-3.10}
	A zero-dimensional space X is a $P$-space if and only if each ideal in $C_c(X)$ is closed in the $m_c$-topology.
\end{theorem}
\begin{proof}
	It follows from Corollary \ref{c-3.9} that each ideal $I$ in $C_c(X)$ is closed in the $m_c$-topology if and only if each ideal in $C_c(X)$ is the intersection of all the maximal ideals in $C_c(X)$ containing it. In view of Corollary 5.7 and Theorem 5.8 in \cite{ref14}, the last condition is equivalent to the requirement that $X$ is a $P$-space.
\end{proof}
	Before examining the Von-Neumann regularity of the intermediate rings in the family $\Sigma_c(X)$, we need to further organize our machinery accordingly. A commutative ring $R$ with unity is called reduced if $0$ is the only nilpotent element of $R$. It is trivial that each $A_c(X)\in \Sigma_c(X)$ is a reduced ring. In what follows all the rings that will appear will be assumed to be reduced. An ideal $I$ (proper) in $R$ is called a $z^0$-ideal in $R$ if for each $a \in I$, $\mathcal{P}_a\subseteq I$, where $\mathcal{P}_a$ is the intersection of all minimal prime ideals in $R$ which contains a. We reproduce the following standard useful formula for the ${\mathcal{P}_a}$ from (\cite{ref7}, Proposition 1.5).

\begin{theorem}\label{t-3.11}
	For each $a\in R$,  $\mathcal{P}_a= \{b\in R:Ann(a)\subseteq Ann(b)\}$, where $Ann(a)=\{c\in R:ac=0\}$ is the annihilator of $a$ in $R$. We also reproduce the following standard proposition.
\end{theorem}

\begin{theorem}\textnormal{(}Due to Kist, \cite{ref16}\textnormal{):}\label{t-3.12}
	A prime ideal $P$ in a ring $R$ is a minimal prime ideal if and only if for each $a\in P$ there exists $b\in R\setminus P$ such that $a.b$ is a nilpotent member of $R$ and in particular $a.b=0$ if the ring $R$ is assumed to to reduced.
\end{theorem}
\begin{remark} \label{r-3.13}
Each element of a minimal prime ideal in $R$ is a divisor of zero. Consequently each element of a $z^0$-ideal in $R$ is a divisor of zero.
\end{remark}

	The following fact is standard and a simple proof is offered in \cite{ref9}~Theorem 4.1.

\begin{theorem}\label{t-3.15}
	Each proper ideal in a Von-Neumann regular ring is a $z^0$-ideal.
\end{theorem}
\begin{theorem}\label{t-3.16}
	An intermediate ring $A_c(X)\in \Sigma_c(X)$ is an absolutely convex subring of $C_c(X)$ in the following sence: If  $|f|\leq|g|$ with $g\in A_c(X)$ and $f\in C_c(X)$ then $f\in A_c(X)$. In particular $A_c(X)$ is a lattice ordered ring.
\end{theorem}

\begin{proof}
	since If $|f|\leq|g|$ it follows that $\frac{f}{1+g^2}$ is a bounded function in $C_c(X)$. Thus $f=\frac{f}{1+g^2}.(1+g^2)\in A_c(X)$.
\end{proof}

The following result tells that no intermediate ring in the family $\Sigma_c(X)\setminus \{C_c(X)\}$ can be ever Von-Neumann regular.
\begin{theorem}\label{t-3.17}
	Suppose $A_c(X)\in \Sigma_c(X) $ is Von-Neumann regular , then $A_c(X)=C_c(X)$.
\end{theorem}
\begin{proof}
	Choose $f\in C_c(X)$. We shall show that $f\in A_c(X)$. Because of the absolute convexity of $A_c(X)$ in $C_c(X)$ in the last theorem it suffices to show that $|f|\in A_c(X)$. We shall indeed show that $\frac{1}{1+|f|}$ is a multiplicative unit of the ring $A_c(X)$ and that will do. Suppose towards a contradiction and let $\frac{1}{1+|f|}$ be not a multiplicative unit of $A_c(X)$. It is clear because of the boundedness of the function $\frac{1}{1+|f|}$ over $X$ that $\frac{1}{1+|f|}\in A_c(X) $. Therefore the principle ideal $<\frac {1}{1+|f|}>=I$ in $A_c(X)$
 generated by this function is a proper ideal and is hence by Theorem \ref{t-3.15} a $z^0$-ideal in $A_c(X)$. It follows from Remark \ref{r-3.13} that $\frac{1}{1+|f|}$ is a divisor of zero in $A_c(X)$ -a contradiction.
\end{proof}
 
 Since a zero-dimensional space $X$ is a $P$-space if and only if $C_c(X)$ is Von-Neumann regular (Corollary 5.7, \cite{ref14}), the following proposition is immediate from the above theorem.
\begin{theorem}
	Let $X$ be a $P$-space . Then $A_c(X)\in \Sigma_c(X)$ is Von-Neumann reular if and only if $A_c(X)=C_c(X)$.
\end{theorem}

\section{Almost $P$-spaces $X$ vis-a-vis the $z^0$-ideals in $A_c(X)$. }
Since the $z^0$-ideals in $A_c(X)$ are all divisors of zero, the following formula to determine them will be needed from time to time.
\begin{theorem}\label{t-4.1}
	An $f\in A_c(X)$ is a divisor of zero in this ring if and only if $Int_X Z(f) \neq\emptyset$.
\end{theorem}
\begin{proof}
	Suppose $f\in A_c(X)$ is a divisor of zero. Then $f\neq0$ and there exists $g\neq 0$ in $A_c(X)$ such that $fg=0$. This shows that $Z(f)\cup Z(g)=X$ and hence $X-Z(g)\subseteq Z(f)$. As $X\setminus Z(g)$ is a non-empty open set in $X$, it follows that $Int_X Z(f)\neq \emptyset$.\\
	
	Conversely let  $Int_X Z(f)\neq \emptyset$. Choose $p$ from this nonempty set. Since $X$ is  zero-dimensional, functions in $C_c(X)$ with their range contained in $[0,1]$ can separate points and closed sets in $X$ (Proposition 4.4, \cite{ref14}). Therefore there exists $g\in C_c(X)$ such that $g(p)=1$ and $g(X\setminus Int_X Z(f))=0$. It is clear that $f.g=0$ and $g\neq0$. Thus $f$ is divisor of zero in $A_c(X)$.
\end{proof}

The next proposition will also be useful to us:
\begin{theorem}\label{t-4.2}
	Let $X$ be zero-dimensional and $f,g\in A_c(X)$. Then $Int_XZ(f)\subseteq Int_XZ(g)$ if and only if $Ann(f)\subseteq Ann(g)$ in the ring $A_c(X)$.
\end{theorem}
\begin{proof}
	Let  $Int_XZ(f)\subseteq Int_XZ(g)$. Choose $h\in Ann(f)$, then $hg=0$. This implies that $X\setminus Z(h)\subseteq Z(f)$, which further implies that $X\setminus Z(h)\subseteq Int_X Z(f)\subseteq Int_X Z(g)\subseteq Z(g)$. Hence $g.h=0$ i.e., $h\in Ann(g)$. Thus $Ann(f)\subseteq Ann(g)$.
	
	Conversely let $Ann(f)\subseteq Ann(g)$. It is sufficient to check that $Int_X Z(f)\subseteq Z(g)$. If possible let there exist a point $p\in Int_XZ(f)\setminus Z(g)$. Since $X$ is zero-dimensional, there exists an $h\in C_c^*(X)\subseteq A_c(X)$ such that $h(p)=1$ and $h(X\setminus Int_XZ(f))=0$. It follows that $h.f=0$ i.e., $h\in Ann(f)$ but $h(p)g(p)\neq0$. So that $h.g\neq0$ and hence $h\notin Ann(g)$. This is a contradiction.
\end{proof}  
 A combination of Theorem \ref{t-3.10} and Theorem \ref{t-4.2} yields the following result:
 \begin{theorem}\label{t-4.3}
 	 	For any $f\in A_c(X)$ $\mathcal{P}_f=\{g\in A_c(X):Ann(f)\subseteq Ann(g)\}=\{g\in A_c(X):Int_X Z(f)\subseteq Int_XZ(g)\}$.
 \end{theorem}

We recall that $\mathcal{P}_f$ is the intersection of all the minimal prime ideals in $A_c(X)$ which contain $f$.
Before taking up the problem of characterizing almost $P$-spaces $X$ via $z^0$-ideals in $C_c(X)$, we need to recall the notion of $z$-ideal in an arbitrary commutative ring $R$ with unity.
\begin{definition}
	An ideal $I$ in $R$ is called a $z$-ideal in $R$ if for each $a\in I$, $M_a\subseteq I$, here $M_a$ is the intersection of all maximal ideals in $R$ containing $a$. Evidently each maximal ideal in $R$ is a $z$-ideal. This notion of $z$-ideal is consistent with the notion of $z$-ideals in $C(X)$. (See 4A, \cite{ref13})
	\\
	
	The following result identifies $z$-ideals and $z_c$-ideals in $C_c(X)$. An ideal $I$ in $C_c(X)$ is called a $z_c$-ideal in \cite{ref14} if whenever $Z(f)\in Z_c(I)=\{Z(g):g \in I\}, f\in C_c(X)$, then $f\in I$.
\end{definition}
\begin{theorem}\label{t-4.5}
	Let $X$ be zero-dimensional . Then an ideal $I$ in $C_c(X)$ is a $z$-ideal if and only if it is a $z_c$-ideal.
\end{theorem}

	\begin{proof}
		Let $I$ be a $z_c$-ideal in $C_c(X)$. Let $f\in I$ and $g\in M_f$, this means that if for $p\in \beta_0X $, $f\in M_c^p$ then $g\in M_c^p$. This implies that $cl_{\beta_0X} Z(f)$ $\subseteq$ $cl_{\beta_0X} Z(g)$, which further implies on taking intersection with $X$ that $Z(f)\subseteq Z(g)$. Since $f\in I$ and $I$ is a $z_c$-ideal in $C_c(X)$ it follows that $g \in I$. Thus $M_f\subseteq I$ and hence $I$ is a $z$-ideal in $C_c(X)$.
		
		Conversely let $I$ be a $z$-ideal in $C_c(X)$, $f\in I$ and $Z(f)\subseteq Z(g) $ with $g\in C_c(X)$. We have to show that $g\in I$. Since $I$ is a $z$-ideal in $C_c(X)$ it suffices to show that $g\in M_f$. So let $M_c^p$ be any maximal ideal in $C_c(X)$ it suffices to show that $g\in M_f$. So let $M_c^p$ be any maximal ideal in $C_c(X)$, $p\in \beta_0X$ which contains $f$, we have to show that $g\in M_c^p$. Indeed $f\in M_c^p$ implies that $p\in cl_{\beta_0 X} Z(f)$ which further implies that $p\in cl_{\beta_0 X} Z(g)$ hence $g\in M_c^p$. Thus altogether $I$ becomes a $z_c$-ideal in $C_c(X)$.
	\end{proof}
We next establish the countable analogue of the well-known fact 3.11(b) in \cite{ref13}.
\begin{theorem}\label{t-4.6}
	Let $K$ be a compact set contained in a $G_\delta$-set $G$ in a zero-dimensional Hausdorff space $X$. Then there exists $Z\in Z_c(X)$ such that $K\subseteq Z\subseteq G$.
\end{theorem}

\begin{proof}
	We can write $G=\bigcap\limits_{n=1}^{\infty}G_n$, where each $G_n$ is open in $X$. Since $K\subset X$ and $K\setminus G_n$ are disjoint closed set in $X$ with $K$ compact, hence by proposition 4.3 in [14], there exists an $f_n\in C_c(X)$ such that $f_n(K)=0$ and $f\in (X\setminus G_n)=1$. This implies that $K\subseteq \bigcap \limits_{n=1}^{\infty}Z(f_n)\subseteq G$. Since $Z_c(X)$ is closed under countable intersection by Lemma 2.2(a) in \cite{ref5}, it follows that $\bigcap\limits_{n=1}^{\infty}Z(f_n)=Z(f)$ for some $f\in C_c(X)$.This implies that $K\subseteq Z(f)\subseteq G$.
\end{proof}

Before seriously embarking on almost $P$-spaces, we introduce the following localized version of this requirement.
\begin{definition}
	A point $p\in X$ is called an almost $P$-point on $X$ if for any zero set $Z$ in $X$ containing $p$, $Int_X Z\neq\emptyset$. Thus $X$ is an almost $P$-space if and only if each point on $X$ is an almost $P$-point.
\end{definition}
\begin{theorem}\label{t-4.8}
	The following statements are equivalent for a point $p$ on a zero-dimensional Hausdorff space  $X$.
	
\begin{enumerate}
	\item  $p$ is an almost $P$-point on $X$.
	\item For any any $G_\delta$-set $G$ containing $p$, $Int_X G\neq\emptyset$.
	\item For any $Z\in Z_c(X)$ containing $p$, $Int_XZ\neq\emptyset$.
\end{enumerate}
\end{theorem}

\begin{proof}
	$(1)\implies(3)$ and $(2)\implies(1)$ are trivial.\\
	$(3)\implies(2)$: Let (3) hold. Let $G$ be a $G_\delta$ set in $X$ containing $p$. Then by Theorem \ref{t-4.6}, there exists $Z\in Z_c(X)$ such that $p\in Z\subset G$. Since $Int_X Z\neq \emptyset$ it follows from (3) that $Int_XG\neq\emptyset$. 
\end{proof}
\begin{corollary}\label{c-4.9}
A zero-dimensional space $X$ is an almost $P$-space if and only if for any nonempty $Z\in Z_c(X)$, $Int_X Z\neq \emptyset$.
\end{corollary}
We are now ready to offer the following comprehensive theorem giving several characterization of almost $P$-space.
\begin{theorem}\label{t-4.10}
The following statements are equivalent for a zero-dimensional space $X$.
\begin{enumerate}
	\item $X$ is almost $P$.
	\item Every maximal ideal in $C_c(X)$ is a $z^0$-ideal.
	\item Every fixed maximal ideal in $C_c(X)$ is a $z^0$-ideal.
	\item Every $z$-ideal in $C_c(X)$ is a $z^0$-ideal.
\end{enumerate}
\end{theorem}

\begin{proof}
	$(1)\implies (2)$: Let $X$ be almost $P$-space and $M$ be a maximal ideal in $C_c(X)$ . Then we can write $M=M_c^p=\{f\in C_c(X): p\in cl_{\beta_0X}Z(f)\}$ for some point $p\in \beta_0X$. Choose $f\in M$ we shall show that $\mathcal{P}_f \subseteq M$ and hence $M$ is a $z^0$-ideal in $C_c(X)$. So let $g\in \mathcal{P}_f$. Then from Theorem \ref{t-4.3} we get $Int_X Z(f)\subseteq Int_XZ(g)$. But since $X$ is almost $P$, each zero set in $X$ is regular closed see \cite{ref17}. This implies that $Z(f)=cl_X(Int_XZ(f))\subseteq cl_X(Int_XZ(g))=Z(g)$. But $f\in M$ implies that $p\in cl_{\beta_0X}Z(f)$, consequently  $p\in cl_{\beta X}Z(g)$ and hence $g\in M_c^p=M$. Thus $\mathcal{P}_f\subseteq M$.\\
	$(2)\implies(3)$: is trivial.\\
	$(3)\implies(1)$: Let (3) be true. It is sufficient to show in view of Corollary \ref{c-4.9} that, for a non-empty $Z\in Z_c(X)$, $Int_XZ\neq\emptyset$. Indeed $Z=Z(f)$ for some $f\in C_c(X)$. Choose a point $p\in Z$, then $f\in M_{p,c}=\{g\in C_c(X):g(p)=0\}$. Now by (3), $M_{p,c}$ is a $z^0$-ideal, consequently by Remark \ref{r-3.13}, $f$ is a divisor of zero in $C_c(X)$. This implies by Theorem \ref{t-4.1} that , $Int_XZ(f)\neq\emptyset$.\\
	$(4)\implies(2)$: is trivial because each maximal ideal in a ring $R$ is a $z$-ideal.\\
	$(1)\implies(4)$: Let $X$ be almost $P$-space and $I$ be a $z$-ideal in $C_c(X)$. Then by Theorem \ref{t-4.5}, $I$ is a $z_c$-ideal in $C_c(X)$. Let $f\in I$ we need to verify that $\mathcal{P}_f\subseteq I$ in order to show that $I$ is a $z^0$-ideal in $C_c(X)$. Choose $g\in \mathcal{P}_f$ then it follows from Theorem \ref{t-4.3} that $Int_XZ(f)\subseteq Int_X Z(g)$. As $X$ is almost $P$ we can therefore write: 
	$Z(f)=cl_X(Int_XZ(f))\subseteq cl_X(Int_XZ(g) )=Z(g)$. Since $f\in I$ and $I$ is a $z_c$-ideal, it follows that $g\in I$. Thus $\mathcal{P}_f\subseteq I$.
\end{proof}

We now show that on choosing $A_c(X)\in \Sigma_c(X)\setminus \{C_c(X)\}$ Theorem \ref{t-4.10} can not be improved by writing that $X$ is almost $P$ if and only if each maximal ideal in $A_c(X)$ is a $z^0$-ideal (respectively each $z$-ideal in $A_C(X)$ is a $z^0$-ideal). 
\begin{theorem}\label{t-4.11}
	Let $A_c(X)$ be an intermediate ring in $\Sigma_c(X) $ properly contained in $C_c(X)$. Then there exists a maximal ideal $M$ in $A_c(X)$ which is not a $z^0$-ideal (clearly $M$ is also a $z$-ideal in $A_c(X)$ which is not a $z^0$-ideal).
\end{theorem}

\begin{proof}
	We select $f\in C_c(X)$ such that $f\notin A_c(X)$. Take $g=\frac{1}{1+|f|}$, then $g\in C_c^*(X)\subseteq A_c(X)$. It follows from absolute convexity of $A_c(X)$ in $C_c(X)$ (Theorem \ref{t-3.16}) that $1+|f|\notin A_c(X)$. Hence $g$ is not invertible in $A_c(X)$. So there exists a maximal ideal $M$ in $A_c(X)$ such that $g\in M$. Since $g$ is not a divisor of zero in $A_c(X)$ (Theorem \ref{t-4.1}). It follows from Remark \ref{r-3.13} that $M$ is not a $z^0$-ideal in $A_c(X)$.
\end{proof}

Theortem \ref{t-4.10} and Theorem \ref{t-4.11} combined together yield the following characterization of $C_c(X)$ among members of $\Sigma _c(X)$.
\begin{theorem}\label{t-4.12}
	Let $X$ be almost $P$. Then the following three statements are equivalent for an $A_c(X)\in \Sigma_c(X)$.
	\begin{enumerate}
		\item Each maximal ideal of $A_c(X)$ is a $z^0$-ideal
		\item Each $z$-ideal of $A_c(X)$ is a $z^0$-ideal
		\item $A_c(X)=C_c(X)$.
	\end{enumerate}
\end{theorem}
Compare with similar kind of characterizations in \cite{ref9},\cite{ref21}, \cite{ref22}.

\section{Pseudocompact spaces $X$ via $U_c$-topologies/$m_c$-topologies on $C_c(X)$}

\begin{Notation}

		For $f\in C_c(X)$ and $\epsilon >0$ in $\mathbb{R}$. Let $U_c(f,\epsilon)=\{g\in C_c(X):  Sup_{x\in X} |f(x)-g(x)|<\epsilon\}$
	\end{Notation}
It is easy to check that the family \{ $U_c(f,\epsilon):f\in C_c(X), \epsilon>0\}$ is an open base for some topology on $C_c(X)$ which we call the $U_c$-topology on $C_c(X)$ and $C_c(X)$ becomes an additive topological group in this topology. The following proposition shows that $C_c(X)$ neither a topological ring nor a topological vector space unless $X$ is pseudocompact. 
\begin{theorem}\label{t-5.2}
	For a zero-dimensional Hausdorff space $X$, the following statements are equivalent:
	\begin{enumerate}
		\item $X$ is pseudocompact.
		\item $C_c(X)$ with $U_c$-topology is a topological ring.
		\item $C_c(X)$ with $U_c$-topology is a topological vector space.
		\item The set $W$ of all units in $C_c(X)$ is open in $U_C$-topology.
	\end{enumerate}	
\end{theorem}

\begin{proof}
	First assume that $X$ is pseudocompact ie; $C_c(X)=C_c^*(X)$. Then the $U_C$-topology on $C_c(X)$ coincides with the uniform norm topology on it and $C_c(X)$ becomes a real normed algebra. It is a standard result in Functional Analysis that a real normed algebra is a topological ring as well as a real topological vector space, where the units $W$ make an open subset of $C_c(X)$.
	Conversely, let $X$ be not pseudocompact. Then there exists an $f\in C_c(X)\setminus C_c^*(X)$ with $f\geqq1$. Take $g=\frac{1}{f}$. Then $g\in C_c^*(X)$ and it takes values arbitrarily near to zero on $X$. We note that for arbitrary $\epsilon>0,$ $ \delta>0$ in $\mathbb{R}$,  $\frac{\epsilon}{\underline{2}}$ (the constant function on $X$ with value $\frac{\epsilon}{2}$)~$\in U_c(0,\epsilon)$ and $f\in U_c(f,\delta)$ while $\frac{\epsilon}{\underline{2}}.f\notin U_c(0,1)$. This proves that the function: $C_C(X)\times C_c(X)\rightarrow C_C(X)$ defined as follows $(k,l)\mapsto k.l$ is not continuous at the point $(0,f)$. It can be proved analogously that the scalar multiplication function:
	\begin{center}
	$\mathbb{R}\times C_c(X)\rightarrow C_c(X)$  
	
    $(r,f)\mapsto r.f$
\end{center}
    is not continuous at $(0,f)$. Thus $C_c(X)$ neither a topological ring nor a topological vector space over $R$.

Finally we observe that $g$ is a unit of $C_c(X)$ i.e; $g\in W$. To show that $W$ is not an open set we shall show that $g$ is not an interior point of $W$. Choose $\epsilon >0$ in $\mathbb{R}$. Since $g$ takes values arbitrarily near to zero on $X$, there exists $a\in X$ such that $0<g(a)<\epsilon$. Take $h=g-g(a)$, then $h\in C_c(X)$ and $h\in U_c(g,\epsilon)$ but $h$ is not a unit of $C_c(X)$ as $h(a)=0$. Thus $U_c(g,\epsilon)$ is not a subset of $W$ and hence $g$ is not an interior point of $W$.
\end{proof}

As in the classical scenario with $C(X)$ (see 2N, \cite{ref13}) it is easy to observe that the relative topology on $C_c^*(X)$ induced by the $m_c$-topology on $C_c(X)$ is finer that the uniform norm topology on $C_c^*(X)$. The following proposition  says that these two topologies coincide when and only when $X$ is pseudocompact.
\begin{theorem}\label{t-5.3}
	The following two statements are equivalent for a Hausdorff zero-dimensional space $X$.
	\begin{enumerate}
		\item $X$ is pseudocompact.
		\item The relative $m_c$-topology on $C_c^*(X)$ is identical to the uniform norm topology on it.
	\end{enumerate}
\end{theorem}
\begin{proof}
	First assume that $X$ is pseudocompact. In view of the above observations, it is sufficient to show that that relative $m_c$-topology $C_c^*(X)$ is weaker than the uniform norm topology. Choose $f\in C_c^*(X)$ and a positive unit $u$ of this ring. Then $u$ is bounded away from zero so that we can write $u(x)\geq \lambda $ for all $x\in X$ for some $\lambda>0$. It follows that the closed ball $\{g\in C_c^*(X)):||f-g||\leq\lambda\}$ centered at $f$ with radius $\lambda$ in the norm topology is contained in $M(f,u)$ and we are through.
	
To prove the converse let $X$ be not pseudocompact. To show that the relative $m_c$-topology on $C_c^*(X)$ is not the same as the uniform norm topology on it, we shall show that $C_c^*(X)$ in the former topology is not a topological vector space, Since $X$ is not pseudocompact there exists $k\in C_c^*(X)$ such that $k$ is a positive unit of $C_c(X)$ which takes values arbitrarily near to zero on X. It follows that there does not exist any pair of distinct real numbers $r,s$ with $|\underline r -\underline s|\leq k$ on $X$. Hence for any $r\in \mathbb{R}$, $M(\underline{r},k)\cap \{\underline s:s\in \mathbb{R}\}=\{\underline{r}\}$ in other words the set of all constant functions in $C_c^*(X)$ is a discrete subset of $C_c^*(X)$ in the relative $m_c$-topology. Consequently the scalar multiplication map: $\mathbb{R}\times C^*(X)\rightarrow C^*(X))$ defined as follows $(r,f)\rightarrow r.f$ is not continuous at the points like $(r,\underline{s})$ with $r,s \in$ $\mathbb{R}$, here $\underline{s}$ stands for for the constant function with value $'s'$ on $X$.

\end{proof}

\section{Questions of Noetherianness/ Artinianness about $C_c(X)$ and their chosen subrings.}

A commutative ring $R$ (with or without identity) is called Noetherian/ Artinian if any ascending sequence of ideals $I_1\subseteq I_2\subseteq .....$/ descending sequence of ideals $I_1\supseteq I_2\supseteq .....$ terminates at a finite stage. It is established in \cite{ref2} that for a Tychonoff space $X$, $C(X)$ (respectively $C^*(X)$) is never Noetherian and also never Artinian unless $X$ is a finite set. Noetherianness/ Artinianness of a selected class of subrings of $C(X)$ are also examined in \cite{ref2}. In the present section our intention is to record the appropriate counterparts of the problems dealt in \cite{ref2} in the context of the rings $C_c(X)$ and $C_c^*(X)$ for a zero-dimensional Hausdorff space $X$.

A family $\mathcal{P}$ of closed set in $X$ is called an ideal of closed sets if 
\begin{enumerate}
	\item$A\in \mathcal{P}, B\in \mathcal{P}\implies A\cup B\in \mathcal{P}$ and 
 	\item $A\in \mathcal{P}$ and $K\subseteq A$ with $K$ closed in $X$ $\implies K\in \mathcal{P}$
\end{enumerate}

\begin{Notation}
Let $\Omega(X)$ stand for the family of all ideals of closed sets in $X$ with $\mathcal{P}\in \Omega(X)$. We associate the following two subrings of $C(X)$:\\
$C_{\mathcal{P}}(X)=\{f\in C(X):cl_X (X\setminus Z(f))\in \mathcal{P}\}$ and\\
$C_{\infty}^{\mathcal{P}}(X)=\{f\in C(X):$ for each $n\in N,\{x\in X:|f(x)|\geqq \frac{1}{n}\}\in \mathcal{P}\}$ with $C_{\mathcal{P}}(X)$ a $z$-ideal in $C(X)$. $X$ is called locally $\mathcal{P}$ if each point $x\in X$ has an open neighbourhood $W$ with its closure lying on $\mathcal{P}$. Thus the local $\mathcal{P}$ condition reduces to local compactness if $\mathcal{P}$ is the ideal of all compact sets in $X$ and in this case $C_\mathcal{P}(X)=C_K(X)$ and $C_{\infty}^\mathcal{P}(X)=C_{\infty}(X)$. For more information on ideal related problems we refer the articles \cite{ref3},\cite{ref4}.

The following result is standard and is recorded in \cite{ref2}, Lemma 2.1.

\end{Notation}

\begin{lemma}\label{l-6.2}
	For any finitely many commutative rings $R_1,R_2,...R_n$ each with identity, ideals of the direct product $R_1\times R_2\times...\times R_n$ are precisely of the form: $J_1\times J_2\times...\times J_n$ where for $j=1,2,...n, \ J_j $ is an ideal in $R_j$.
\end{lemma}

We record the following convenient version of the local $\mathcal{P}$ condition for a zero-dimensional space $X$.
\begin{theorem}\label{t-6.3}
	For a zero-dimensional Hausdorff space $X$, the following statements are equivalent:\begin{enumerate}
		\item $X$ is locally $\mathcal{P}$
		\item $\{Z(f):f\in C_\mathcal {P}(X)\cap C_c(X)\}$ is a base for the closed sets in $X$.
		\item  $\{Z(f):f\in C_\infty^\mathcal {P}(X)\cap C_c(X)\}$ is a closed base for $X$.
		\item  $\{Z(f):f\in C_\mathcal {P}(X)\cap C_c^*(X)\}$ is a closed base for $X$.
		\item $\{Z(f):f\in C_\infty^\mathcal {P}(X)\cap C_c^*(X)\}$ is a closed base for $X$.
	\end{enumerate}
\end{theorem}

We omit the proof of this theorem, which can be done by closely following the arguments and making some necessary modifications in the proof of Theorem 4.3 in \cite{ref2}. We are now ready to enunciate the main theorem in this section.
\begin{theorem}\label{t-6.4}
	Let $\mathcal{P}\in \Omega(X)$ where $X$ is a zero-dimensional Hausdorff space which is further locally $\mathcal{P}$. Then the following statements are equivalent:\begin{enumerate}
		\item$ C_\mathcal {P}(X)\cap C_c(X)$ is a Noetherian Ring.
		\item$ C_\mathcal {P}(X)\cap C_c(X)$ is an Artinian Ring.
		\item $C_\infty^\mathcal {P}(X)\cap C_c(X)$ is a Noetherian Ring.
		\item $C_\infty^\mathcal {P}(X)\cap C_c(X)$ is an Artinian Ring.
		\item $X$ is a finite set.
	\end{enumerate}
\end{theorem}
 The proof can be accomplished by making a close introspection of the reasonings made in the proof of the Theorem~ 1.1 in \cite{ref2}. Nevertheless to make the article self-contained and to highlight a few important remarks regarding the possible dearth of Noetherian Rings/ Artinian rings lying between $C_c^*(X)$ and $C_c(X)$, we wish to provide an alternatively framed regorous proof of the above theorem.\\

 \noindent Proof of the Theorem \ref{t-6.4}:
 First assume that $X$ is a finite set with $'n'$ elements. Then since $X$ is Hausdorff it becomes a discrete space. Consequently $C(X)=\mathbb{R}^X$, which is isomorphic to direct product of $\mathbb{R}\times \mathbb{R}\times ...\times \mathbb{R}$ ($n$ times). On the other hand since $X$ is locally $\mathcal{P}$ it follows that $C_\mathcal{P}(X)=C_\infty^\mathcal{P}(X)=C_c(X)=C(X)$ consequently $C_\mathcal{P}(X)\cap C_c(X)=C_\infty^\mathcal{P}(X)\cap C_c(X)=C_c(X)=C(X)$. Since the field $\mathbb{R}$ has just 2 ideals, it follows from Lemma \ref{l-6.2} that C(X) has just $2^n$ many ideals. Hence the rings $C_\mathcal{P}(X)\cap C_c(X)$ and $C_\infty^\mathcal{P}(X)\cap C_c(X)$ are both Noetherian and Artinian.
 
 Conversely, let $X$ be an infinite set. We shall show that thae ring $C_\mathcal{P}(X)\cap C_c(X)$ is not a Noetherian ring. Analogous arguments can be made to show that $C_\mathcal{P}(X)\cap C_c(X)$ is not an Artinian ring and nor is the ring  $C_\infty^\mathcal{P}(X)\cap C_c(X)$ Noetherian or Artinian. As $X$ is an infinite Hausdorff space it contains a copy of $\mathbb{N}$ (0.13, \cite{ref13}), So for each $k\in \mathbb{N}$ there exists an open set $W_k$ in $X$ such that $W_k\cap \mathbb{N}=\{k\}$. Since $X$ is locally $\mathcal{P}$ and zero-dimensional, we can employ Theorem \ref{t-6.3} to find for each $k\in \mathbb{N}$, an $f_k\in  C_\mathcal{P}(X)\cap C_c(X)$ such that $k\in X\setminus Z(f_k)\subset W_k$....(1). We now assert that the ideal $I=<f_1,f_2,...f_k,..>$ generated by these $f_k's$ in the ring  $C_\mathcal{P}(X)\cap C_c(X)$ can not be finitely generated and hence $C_\mathcal{ P}(X)\cap C_c(X)$ is not Noetherian.(A ring $R$ is Noetherian if and only if each ideal in $R$ is finitely generated: A standard result).\\
 
 \noindent Proof of the assertion: Choose $n\in N$. We show that the ideal $<f_1,f_2,...f_n>\varsubsetneqq I$ and that will do. Indeed from (0) and (1) it follows that $f_{n+1}(n+1)\neq 0 $, while $f_1(n+1)=f_2(n+1)=...=f_n(n+1)=0$. Thus there do not exist functions $l_1,l_2,...l_n\in  C_\mathcal{P}(X)\cap C_c(X)$ for which we can write: $ f_{n+1}=l_1f_1+l_2f_2+...l_nf_n$. Hence $f_{n+1}\in I \  \setminus  <f_1,f_2,...f_k,..>$.

\begin{remark}\label{r-6.5}
	Since for any $\mathcal{P}\in \Omega(X)$, $C_\mathcal{P}(X) \subseteq C_\infty^\mathcal{P}(X)$ an easy verification, it follows from Theorem \ref{t-6.3} that for any prescribed ring $R$ lying either between $C_\mathcal{ P}(X)\cap C_c(X)$ and $C_\infty^\mathcal{ P}(X)\cap C_c(X)$ or between $C_\mathcal{ P}(X)\cap C_c(X)$ and $C_\mathcal{ P}(X)\cap C_c^*(X)$, a zero-dimensional space $X$ is locally $\mathcal{P}$ if and only if $\{Z(f):f\in R\}$ is a base for the closed sets in $X$. With this observation in mind, if we make a close scrutiny into the proof of the converse part of Theorem \ref{t-6.4}, we get the following result.
\end{remark}

\begin{theorem}\label{t-6.6}
	Given $\mathcal{P}\in \Omega (X)$, if $X$ is an infinite zero-dimensional locally $\mathcal{P}$ space, then no ring lying between $C_\mathcal{ P}(X)\cap C_c(X)$ and $C_\infty^\mathcal{ P}(X)\cap C_c(X)$ is Noetherian (respectively Artinian) and also no ring lying between $C_\mathcal{P}(X)\cap C_c(X)$ and $C_\mathcal{ P}(X)\cap C_c^*(X)$ is Noetherian (respectively Artinian).
\end{theorem}

We record below two special cases of Theorem \ref{t-6.6}, on choosing $\mathcal{P}\equiv$ the ideal of all compact sets in $X$ in the first part of the Theorem and on choosing $\mathcal{P}\equiv$ the ideal of all closed sets in $X$ in the second part of the theorem.
\begin{theorem}\label{t-6.7}
	\begin{enumerate}
		\item If $X$ is an infinite locally compact zero-dimensional space then no ring lying between $C_K(X)\cap C_c(X)$ and $C_\infty(X)$ $\cap C_c^*(X)$ is Notherian/ Artinian.
		
		\item For any infinite zero-dimensional space $X$, no intermediate ring $A_c(X)\in \Sigma_c(X)$ is Notherian/ Artinian.
	\end{enumerate} 
\end{theorem}
\section{Formula for $z^0$-ideals in intermediate rings.}

We first show that ideal $I$ in an intermediate ring $A_c(X)\in \Sigma(X)$ gives rise to an ideal of closed sets in $X$. Indeed fer any such $I$, we get
$\mathcal{P}_I^{A_c}=\{E\subseteq X:$ $E$ is closed in $X$ and there exists $f\in I$ such that $E\subseteq cl_X(X\setminus Z(f))$. If is easy to verify that $\mathcal{P}_I^{A_c}$ is an ideal of closed sets in $X$ i.e; $\mathcal{P}_I^{A_c}\in \Omega(X)$ and also that $I\subseteq C_{\mathcal{P}_I^{A_c}}(X)\cap A_c(X)$ $ \equiv $ $\{f\in A_c(X):cl_X(X\setminus Z(f))\in \mathcal{P}_I^{A_c} \}$. The following fact tells decisively when does equality occur in the last inclusion relation. Incidentally we get an explicit formula for $z^0$-ideals in the intermediate rings.
\begin{theorem}\label{t-7.1}
	Let $A_c(X)\in \Sigma_c(X)$. Then an ideal $I$ in $A_c(X)$ is a $z^0$-ideal in this ring if and only if there exists $\mathcal{P}\in \Omega(X)$ such that $I=C_{\mathcal{P}}(X)\cap A_c(X)$.
\end{theorem}

\begin{proof}
	First assume that $I$ is a $z^0$-ideal in $A_c(X)$. In view of the observations foregoing this theorem, it is sufficient to show that  $C_{\mathcal{P}_I^{A_c}}(X)\cap A_c(X)\subseteq I$. So let $g\in  C_{\mathcal{P}_I^{A_c}}(X)\cap A_c(X)$ then $cl_X(X\setminus Z(g))\in {\mathcal{P}_I^{A_c}}$. Consequently there exists $f\in I$ such that  $cl_X(X\setminus Z(g))\subseteq cl_X(X\setminus Z(f))$. This implies on taking complement in $X$ that $Int_XZ(g)\supseteq Int_XZ(f)$, which further implies in view of Theorem \ref{t-4.3} that $g\in \mathcal{P}_f\equiv$ the intersection of all minimal prime ideals in $A_c(X)$ containing $f$. Since $f\in I$ and $I$ is a $z^0$-ideal in $A_c(X)$ it follows that $g\in I$. Thus we get: $I= C_{\mathcal{P}_I^{A_c}}(X)\cap A_c(X)$.
	
	To prove the other part of the theorem we show that for any $\mathcal{P}\in \Omega (X)$, $C_{\mathcal{P}}(X)\cap A_c(X)$ is a $z^0$-ideal in $A_c(X)$. Choose $f\in C_{\mathcal{P}}(X)\cap A_c(X)$ , then $cl_X(X\setminus Z(f))\in \mathcal{P}$. We need to verify that $\mathcal{P}_f\subseteq C_\mathcal{ P}(X)\cap A_c(X)$. So choose $g\in \mathcal{P}_f$, then by Theorem \ref{t-4.3} $Int_XZ(f)\subseteq Int_X(g)$, which implies obviously that  $cl_X(X\setminus Z(g))\subseteq  cl_X(X\setminus Z(f))$. Since $f\in C_\mathcal{P}(X)$ it follows that  $cl_X(X\setminus Z(f))\in \mathcal{P}$. As $\mathcal{P}$ is an ideal of closed sets in $X$, this further implies that $ cl_X(X\setminus Z(g))\in \mathcal{P}$ i.e; $g\in C_\mathcal{P}(X)\cap A_c(X)$. Thus $\mathcal{ P}_f \subseteq  C_{\mathcal{P}}(X)\cap A_c(X)$.
\end{proof}

 It is established recently in \cite{ref1}, Theorem 5.2 that an ideal $I$ in $C(X)$ with $X$, Tychonoff is a $z^0$-ideal in $C(X)$ if and only if there exists $\mathcal{ P}\in \Omega(X)$ such that $I=C_\mathcal{ P}(X)$. Therefore we can make the following comments.
 
 \begin{remark} \label{r-7.2}
 	$z^0$-ideals in the intermediate rings $A_c(X)\in \Sigma_c(X)$ with $X$, zero-dimensional are exactly the contractions of $z^0$-ideals in $C(X)$.
 \end{remark}

\end{document}